\documentclass[12pt]{amsart}

\usepackage{amssymb,amsbsy}
\usepackage{latexsym}
\usepackage{verbatim,color}
\usepackage{fullpage,euscript}

\usepackage{xspace}
\usepackage{xcolor}
\usepackage[colorlinks=true,linkcolor=blue,urlcolor=my_color,citecolor=magenta]{hyperref}

\definecolor{my_color}{rgb}{0.5,0.4,0.1}
\definecolor{MIXT}{rgb}{0.7,0.1,0.2}

\numberwithin{equation}{section}

\input {cyracc.def}
 \font\tencyr=wncyr10 %scaled \magstephalf
\font\tencyi=wncyi10 %scaled \magstephalf
\font\tencysc=wncysc10 %scaled \magstephalf
\def\rus{\tencyr\cyracc}
\def\rusi{\tencyi\cyracc}
\def\rusc{\tencysc\cyracc}

\newtheorem{thm}{Theorem}[section]

\newtheorem*{qtn}{Problems}
\newtheorem{lm}[thm]{Lemma}%[chapter]
\newtheorem{cl}[thm]{Corollary}%[chapter]
\newtheorem{prop}[thm]{Proposition}%[chapter]

\theoremstyle{remark}
\newtheorem{rmk}[thm]{Remark}

\theoremstyle{definition}

\newtheorem{exs}[thm]{Examples}
\newtheorem{df}{Definition}

\newcommand {\ce}{{\mathfrak c}}

\newcommand {\g}{{\mathfrak g}}
\newcommand {\h}{{\mathfrak h}}

\newcommand {\ka}{{\mathfrak k}}
\newcommand {\el}{{\mathfrak l}}

\newcommand {\q}{{\mathfrak q}}
\newcommand {\rr}{{\mathfrak r}}

\newcommand {\te}{{\mathfrak t}}

\newcommand {\z}{{\mathfrak z}}

\newcommand {\sln}{{\mathfrak{sl}_n}}

\newcommand {\slno}{{\mathfrak{sl}_{n+1}}}

\newcommand {\spn}{{\mathfrak {sp}}_{2n}}

\newcommand {\son}{{\mathfrak{so}_n}}

\newcommand {\ca}{{\mathcal A}}

\newcommand {\cf}{{\mathcal F}}
\newcommand {\ck}{{\mathcal K}}

\newcommand {\cs}{{\mathcal S}}

\newcommand {\BZ}{{\mathbb Z}}
\newcommand {\BN}{{\mathbb N}}

\newcommand {\md}{/\!\!/}
\newcommand {\isom}{\stackrel{\sim}{\longrightarrow}}

\newcommand {\eus}{\EuScript}

\newcommand {\gN}{{\eus N}}

\newcommand {\zq}{{\eus Z(\q)}}
\newcommand {\zg}{{\eus Z(\g)}}
\newcommand {\zk}{{\eus Z(\ka)}}

\newcommand{\ap}{\alpha}

\renewcommand{\le}{\leqslant}
\renewcommand{\ge}{\geqslant}

\newcommand {\ad}{{\mathrm{ad\,}}}
\newcommand {\ads}{{\mathrm{ad}^*}}

\newcommand {\Ann}{{\mathrm{Ann\,}}}

\newcommand {\codim}{{\mathrm{codim\,}}}

\newcommand {\gr}{{\mathrm{gr\,}}}

\newcommand {\ind}{{\mathrm{ind\,}}}
\newcommand {\Lie}{{\mathrm{Lie\,}}}

\newcommand {\Mor}{\operatorname{Mor}}

\newcommand {\rk}{{\mathrm{rk\,}}}

\newcommand {\spe}{{\mathsf{Spec\,}}}

\newcommand {\trdeg}{{\mathrm{trdeg\,}}}

\newcommand {\beq}{\begin{equation}}
\newcommand {\eeq}{\end{equation}}

\newcommand {\tri}{\mathfrak{sl}_2}
\newcommand {\GR}[2]{{\textrm{{\bf #1}}}_{#2}}

\newcommand {\ov}{\overline}
\newcommand {\un}{\underline}

\newcommand {\bbk}{\Bbbk}%{\mathbb F}%
\newcommand {\sat}{\mathsf{Sat}(\g,\g_0)}
\newcommand {\sath}{\mathsf{Sat}(\h,\h_0)}

\begin{document}
\setlength{\parskip}{3pt plus 2pt minus 0pt}
\hfill {\scriptsize December 6, 2013}
\vskip1ex

\title[On maximal commutative subalgebras]
{On maximal commutative subalgebras of Poisson algebras associated with 
involutions of semisimple Lie algebras}
\author[D.\,Panyushev]{Dmitri I.~Panyushev}
\address[D.P.]%
{Institute for Information Transmission Problems of the R.A.S., %Russian Academy of Sciences, 
\hfil\break\indent \quad Bol'shoi Karetnyi per. 19, Moscow 127994, Russia
}
\email{panyushev@iitp.ru}
\author[O.\,Yakimova]{Oksana S.~Yakimova}
\address[O.Y.]{Mathematisches Institut, Friedrich-Schiller-Universit\"at, \hfil\break\indent
\quad D-07737 Jena,  Deutschland}
\email{oksana.yakimova@uni-jena.de}
\begin{abstract}
For any involution $\sigma$ of a semisimple Lie algebra $\g$, one constructs a non-reductive Lie 
algebra $\ka$, which is called a $\BZ_2$-contraction of $\g$. In this paper, we attack the problem of 
describing maximal commutative subalgebras of the Poisson algebra $\mathcal S(\ka)$. This is closely
related to the study of the coadjoint representation of $\ka$ and the set, $\ka^*_{reg}$, of the regular elements of $\ka^*$. By our previous results, in the context of $\BZ_2$-contractions,
the argument shift method provides maximal commutative 
subalgebras of $\mathcal S(\ka)$ whenever $\codim(\ka^*\setminus \ka^*_{reg})\ge 3$. 
%has the complement of codimension $\ge 3$.

Our main result here is that $\codim(\ka^*\setminus \ka^*_{reg})\ge 3$ if and only if the Satake diagram of
$\sigma$ has no trivial nodes. (A node is trivial, if it is white, has no arrows attached, and all adjacent 
nodes are also white.) The list of suitable involutions is provided. We also describe certain maximal commutative subalgebras of $\mathcal S(\ka)$
if the $(-1)$-eigenspace of $\sigma$ in $\g$ contains regular elements.
\end{abstract}
\maketitle

\section*{Introduction}

\noindent
Let $Q$ be a connected algebraic group, with Lie algebra
$\q$, over an algebraically closed field $\bbk$ of characteristic zero.
The symmetric algebra $\cs(\q)\simeq\bbk[\q^*]$ is equipped with the standard Lie-Poisson bracket $\{\ ,\ \}$, and 
the algebra of invariants $\cs(\q)^Q$ is the centre of $(\cs(\q), \{\ ,\ \})$. We say that a subalgebra 
$\ca\subset \cs(\q)$ is {\it commutative\/} if the bracket $\{\ ,\ \}$ vanishes on $\ca$. 
As is well known, a commutative subalgebra cannot have the transcendence degree larger than
$(\dim\q+\ind\q)/2$, where $\ind\q$ is the index of $\q$. If this bound is attained, then $\ca$ is said to be 
{\it of maximal dimension}. A commutative subalgebra $\ca$ is said to be {\it maximal\/}, 
if it is not contained in a larger commutative subalgebra of $\cs(\q)$.
It is shown in~\cite{mf3} that natural commutative subalgebras of $\cs(\q)$ can be constructed through
the use of $\cs(\q)^Q$ and any $\xi\in\q^*$. This procedure is known as the "argument shift 
method", see Section~\ref{subs:shift} for details. We write $\cf_\xi(\cs(\q)^Q)$ for the resulting commutative 
subalgebra of $\cs(\q)$.

It was proved in \cite{mf3} that if $\q=\g$ is semisimple and $\xi\in\g^*\simeq\g$ is regular 
semisimple, then $\cf_\xi(\cs(\g)^G)$ is of maximal dimension. (Later on, it was realised that these 
subalgebras are also maximal~\cite{tar1}.) Let $\q^*_{reg}$ be the set of $Q$-regular elements of 
$\q^*$. By~\cite{bol2}, if $\trdeg(\cs(\q)^Q)=\ind\q$ and 
$\codim (\q^*\setminus\q^*_{reg})\ge 2$, then $\cf_\xi(\cs(\q)^Q)$ is of maximal dimension for any 
$\xi\in\q^*_{reg}$. In \cite{do},  we extended this result by proving that the maximality of 
$\cf_\xi(\cs(\q)^Q)$  is related to the property that $\codim (\q^*\setminus\q^*_{reg})\ge 3$, 
see also Theorem~\ref{thm:cod3} for the precise statement.

An important class of non-reductive Lie algebras consists of $\BZ_2$-contractions of semi\-simple 
Lie algebras $\g$. A {\it $\BZ_2$-contraction} of $\g$ is the 
semi-direct product $\ka=\g_0\ltimes\g_1$, where $\g=\g_0\oplus\g_1$ is a $\BZ_2$-grading of $\g$ and $\g_1$ becomes an abelian ideal in $\ka$.
All $\BZ_2$-contractions $\ka$ satisfy the property that $\trdeg(\cs(\ka)^K)=\ind\ka$ and 
$\codim (\ka^*\setminus\ka^*_{reg})\ge 2$~\cite{coadj}.
In this paper, we attack the problem of describing maximal commutative subalgebras in $\eus S(\ka)$. 
Our first approach is to verify when Theorem~\ref{thm:cod3} applies to $\ka$.  
To this end, it suffices to check that $\ka^*\setminus\ka^*_{reg}\ge 3$.
Our main result is that such ``{\sl codim}--3 property''  
can be characterised in terms of the Satake diagram of the corresponding
involution of $\g$. Namely,  $\ka=\g_0\ltimes\g_1$ has the {\sl codim}--3 property
if and only if each white node of the Satake diagram has either a black adjacent node or an arrow 
attached (Theorem~\ref{thm:main}). See also Table~\ref{table:all} for the list of relevant involutions and 
Satake diagrams. Thus, for those $\BZ_2$-contractions, the commutative subalgebras
$\cf_\xi(\cs(\ka)^K)$, $\xi\in\ka^*_{reg}$, are maximal. 
Quite a different approach works if $\g_1$ contains a regular nilpotent element of $\g$. 
Here we prove that $\cs(\ka)^{\g_1}$ is a maximal commutative subalgebra (Theorem~\ref{thm:N-reg}). 
An interesting feature is that, for all cases, both constructions provide maximal commutative 
subalgebras of $\eus S(\ka)$ that are polynomial. 
Unfortunately, these results do not cover all $\BZ_2$-contractions. On the other hand, 
there is an involution of $\g=\mathfrak{sl}_{2n+1}$, with 
$\g_0=\mathfrak{sl}_{n}\dotplus\mathfrak{sl}_{n+1}\dotplus\te_1$,  where both approaches apply and
the resulting commutative subalgebras appear to be rather different.

In Section~\ref{prelim}, we gather basic facts on coadjoint representations and commutative 
subalgebras of $\eus S(\q)$, including our sufficient condition for the maximality of subalgebras of the 
form $\cf_\xi(\cs(\q)^Q)$.  Necessary background %invariant-theoretic results  
on the isotropy representations of symmetric spaces and Satake diagrams is presented in 
Section~\ref{sect:sym-pairs}. In Section~\ref{sect:new}, we recall basic properties of $\BZ_2$-contractions 
and prove that the subalgebra $\eus S(\ka)^{\g_1}$ is maximal and 
commutative if and only if $\g_1$ contains a regular nilpotent element of $\g$. Section~\ref{sect:codim3} is devoted to our characterisation of the involutions of semisimple Lie algebras having the property that
$\codim (\ka^*\setminus\ka^*_{reg})\ge 3$. Finally, in Section~\ref{sect:finish}, we summarise our knowledge on maximal commutative subalgebras of $\BZ_2$-contractions and pose open problems.

{\sl \un{Notation}.}
If %an algebraic group 
$Q$ acts on an irreducible affine variety $X$, then $\bbk[X]^Q$ 
is the algebra of $Q$-invariant regular functions on $X$ and $\bbk(X)^Q$
is the field of $Q$-invariant rational functions. If $\bbk[X]^Q$
is finitely generated, then $X\md Q:=\spe \bbk[X]^Q$, and
the {\it quotient morphism\/} $\pi: X\to X\md Q$ is the mapping associated with
the embedding $\bbk[X]^Q \hookrightarrow \bbk[X]$.

If $V$ is a $Q$-module and $v\in V$, then $\q_v$ is the stabiliser of 
$v$ in $\q$.
%For the adjoint representation of $\q$, the stabiliser of $x\in \q$ is also denoted by 
%$\z_\q(x)$, and we say that $\z_\q(x)$ is the {\it centraliser\/} of $x$.
We write $\te_n$ for the Lie algebra of an $n$-dimensional torus and 
use `$\dotplus$' to denote the direct sum of Lie algebras.

%%%%%%%%%%%%%%%%%%%%%
\section{Preliminaries on the coadjoint representation and commutative subalgebras} 
\label{prelim}

\noindent
Let $Q$ be an affine algebraic group with Lie algebra $\q$. If $Q$ acts regularly on an irreducible
algebraic variety $X$, then we also write $(Q:X)$ for this.
Let $X_{reg}$ be the set of $Q$-{\it regular\/} (= $\q$-{\it regular}) elements of $X$. That is,
\begin{multline*}
   X_{reg}:=
   \{ x\in X\mid \dim Q{\cdot}x \ge \dim Q{\cdot}x' \text{ for all } x' \in X\}= \\
    \{ x\in X\mid \dim \q_x \le \dim \q_{x'} \text{ for all } x' \in X\} \ .
\end{multline*}
As is well-known, $X_{reg}$ is a dense open subset of $X$, hence $\codim_X (X\setminus X_{reg})\ge 1$.
%If we want to explicitly specify the group acting on $X$, we refer to $Q$-{\it regular\/} elements.

\begin{df}   \label{def:codim-n}
We say that $(Q:X)$ has the {\it codim--$n$ property\/} if
$\codim_X (X\setminus X_{reg})\ge n$.
\end{df}

\subsection{The coadjoint representation} Henceforth, we assume that $Q$ is connected.
%\noindent
There are two natural representations ($Q$-modules) associated with $Q$, the adjoint and 
coadjoint ones. Accordingly, we write $\q^*_{reg}$ for the set of $Q$-regular elements of $\q^*$,
with respect to the coadjoint representation. 

\begin{df}   \label{def:codim2}
We say that  $\q$ has the {\it codim--$n$ property\/} if $\codim (\q^*\setminus \q^*_{reg})\ge n$, 
i.e., if the {\bf coadjoint} representation (action) of $Q$ has the {\sl codim}--$n$ property.
\end{df}

This notion will mostly be used with $n=2$ or $3$.  

\begin{exs}  \label{ex:codim3-red}
1) If $\q$ is semisimple, then $\ad\simeq\ads$ and $\codim (\q\setminus \q_{reg})=3$, see \cite{ko63}. 
If $\q$  is toral (= Lie algebra of a torus), then $\q^*_{reg}=\q^*$. This also implies that  
all reductive Lie algebras have the {\sl codim}--$3$ property. (We assume that $\codim(\varnothing)=-\infty$.)
\\  \indent
2)  If $\q$ is abelian, then $\q^*_{reg}=\q^*$ and thereby
$\q$ has the {\sl codim}--$n$ property for all $n$.
\\  \indent
3)  For each $n\in\BN$ there exist non-commutative Lie algebras with {\sl codim}--$n$ property 
\cite[Example\,1.1]{do}.
\end{exs}

If $\xi\in \q^*_{reg}$, then $\dim\q_\xi$ is called the {\it index\/} of $\q$, denoted $\ind\q$. In other words, $\ind\q$ is the minimal codimension of $Q$-orbits in $\q^*$.
By Rosenlicht's theorem, we have $\trdeg\bbk(\q^*)^Q=\ind\q$. In particular, 
$\trdeg\bbk[\q^*]^Q\le \ind\q$. Set $b(\q)=(\dim\q+\ind\q)/2$.
If $\q$ is semisimple, then $b(\q)$ is the dimension of  a Borel subagebra.

Let $\cs(\q)\simeq \bbk[\q^*]$ be the symmetric algebra of $\q$
(= the algebra of polynomial functions on $\q^*$).
For $f\in \cs(\q)$,  the differential of $f$, $\textsl{d}f$, is a  
polynomial mapping from $\q^*$ to $\q$, i.e., an element of
$\Mor(\q^*,\q)\simeq \cs(\q)\otimes \q$. 
More precisely, if $\deg f=d$ (i.e., $f\in \cs^d(\q)$), then
$\textsl{d}f\in \cs^{d-1}(\q)\otimes \q$. 
We write $(\textsl{d}f)_\xi$ for 
the value of $\textsl{d}f$ at $\xi\in\q^*$, and
the element $(\textsl{d}f)_\xi\in \q$ is defined as follows. 
If $\nu\in \q^*$ and $\langle\ ,\ \rangle$ is the natural pairing
between $\q$ and $\q^*$, then 
\[
\langle (\textsl{d}f)_\xi,\nu\rangle:= \text{the coefficient of
$t$ in the Taylor expansion of $f(\xi+t\nu)$}.
\] 
The Lie-Poisson bracket in $\cs(\q)$ is defined by  
$\{f_1,f_2\}(\xi)=\langle [(\textsl{d}f_1)_\xi, (\textsl{d}f_2)_\xi], \xi\rangle$ for $\xi\in\q^*$.
Since $Q$ is connected, the algebra of invariants
$\cs(\q)^Q=\cs(\q)^\q$ is the centre of $(\cs(\q), \{\ ,\ \})$.
We also write $\eus Z(\q)$ for the Poisson centre of $\cs(\q)$.

{\bf Warning.} The symbols $\cs(\q)^Q$, $\zq$, and $\bbk[\q^*]^Q$ refer to one and the same algebra.
But we prefer to use $\cs(\q)^Q$ and $\zq$ (resp. $\bbk[\q^*]^Q$) in the Poisson-related 
(resp. invariant-theoretic) context.

From the invariant-theoretic point of view, the usefulness of the {\sl codim}--$2$ property 
is clarified by the following result, see \cite[Theorem\,1.2]{coadj}.

\begin{thm}  \label{thm:cod2}
Suppose that $\q$ has the {\sl codim}--$2$ property
and $\trdeg \bbk[\q^*]^Q=\ind\q$. Set  $l=\ind\q$. 
Let $f_1,\dots,f_l\in \bbk[\q^*]^Q$ be arbitrary
homogeneous algebraically independent polynomials. Then 
\begin{itemize}
\item[\sf (i)] \ $\sum_{i=1}^l\deg f_i \ge b(\q)$; %(\dim\q+\ind\q)/2$;
\item[\sf (ii)] \  If \  $\sum_{i=1}^l\deg f_i = b(\q)$, then
\begin{itemize}
\item \ $\bbk[\q^*]^Q$ is freely generated by $f_1,\dots,f_l$ and
\item \ $\xi\in\q^*_{reg}$ if and only if $(\textsl{d}f_1)_\xi,\dots,(\textsl{d}f_l)_\xi$ are linearly
independent.
\end{itemize}
\end{itemize}
\end{thm}%
\noindent
The second assertion in (ii) is a generalisation of Kostant's result for
reductive Lie algebras \cite[(4.8.2)]{ko63}.

\subsection{Commutative subalgebras of $\cs(\q)$}  \label{subs:shift}
Let $\eus A$ be a subalgebra of the symmetric algebra $\cs(\q)$. Then $\eus A$ is said to be 
{\it commutative} if the restriction of $\{\ ,\ \}$ to $\eus A$ is zero. 

By definition, the transcendence degree of $\eus A$ is that of the quotient field of $\eus A$.
It is well-known that if $\eus A$ is commutative, then
$\trdeg \eus A\le b(\q)$. Indeed, if $f_1,\dots, f_n\in \eus A$ are algebraically independent, then 
for a generic $\xi\in\q^*_{reg}$, the linear span of
$(\textsl{d}f_1)_\xi,\dots,(\textsl{d}f_n)_\xi$ is $n$-dimensional, and it is an isotropic space with respect to 
the Kirillov-Kostant form ${\mathcal K}_\xi$ on $\q$.
(Recall that $\ck_\xi(x,y):=\langle \xi, [x,y]\rangle$ and hence $\dim(\ker\ck_\xi)=\dim\q_\xi=\ind\q$.)

A commutative subalgebra of $\cs(\q)$ is said to be {\it of maximal dimension}, if its 
transcendence degree equals $b(\q)$. A commutative subalgebra of $\cs(\q)$ is  {\it 
maximal}, if it is maximal with respect to inclusion among all commutative subalgebras. It is known 
that commutative subalgebras of maximal dimension always exist, see \cite{sadetov}. 
It is plausible that if $\q$ is algebraic and $\trdeg\eus Z(\q)= \ind\q$, then any maximal commutative 
subalgebra of $\cs(\q)$ is of maximal dimension. 
(In \cite[2.5]{fonya09}, Ooms provides an example of a maximal commutative subalgebra of $\cs(\q)$ that is not of maximal dimension. In his example, $\q$ is not algebraic, which can easily be mended. However,  even after that modification one still has $\trdeg\eus Z(\q)< \ind\q$.)

Let $f\in\cs(\q)$ be a  polynomial of degree $d$.
For  $\xi\in\q^*$, consider a shift of $f$ in direction $\xi$:
$f_{a,\xi}(\mu)=f(\mu+a\xi)$, where $a\in \bbk$.
Expanding the right hand side as polynomial in $a$, we obtain the expression
$f_{a,\xi}(\mu)=\sum_{j=0}^d f_{\xi}^j(\mu)a^j$ and
%Associated with this shift of argument, we obtain 
the family of polynomials
$ f_{\xi}^j$, where $j=0,1,\dots,d-1$. (Since $\deg f_{\xi}^j=d-j$, the value
$j=d$ is not needed.)
%We will say that the polynomials $\{ f_{\xi}^j\}$ are $\xi$-{\it shifts\/} of $f$.
Notice that $f_\xi^0=f$ and $f_\xi^{d-1}$ is a linear form on $\q^*$, i.e.,
an element of $\q$. Actually, $f_\xi^{d-1}=(\textsl{d}f)_\xi$.
There is also an obvious symmetry with respect to $\xi$ and $\mu$: \ 
$f_\xi^j(\mu)=f_\mu^{d-j}(\xi)$.

The following observation is due to Mishchenko--Fomenko \cite{mf3}.

\begin{lm}     \label{mf:shift}
Suppose that $h_1,\dots, h_m\in \eus Z(\q)$ are homogeneous. 
Then for any $\xi\in\q^*$, all the polynomials
$
  \{(h_i)_{\xi}^j \mid i=1,\dots,m;\ j=0,1,\dots,\deg h_i-1\}$
pairwise commute with respect to the Lie--Poisson bracket.
\end{lm}

This procedure has been used for constructing
commutative subalgebras of maximal dimension in $\cs(\q)$.
Given $\xi\in\q^*$ and an arbitrary subset $\eus B\subset \zq$,
let  $\cf_\xi(\eus B)$ denote the subalgebra of $\cs(\q)$ generated 
by the $\xi$-shifts of all elements of $\eus B$. If $\hat{\eus B}$ is the subalgebra
generated by $\eus B$, then $\cf_\xi(\eus B)=\cf_\xi(\hat{\eus B})$.
By Lemma~\ref{mf:shift}, any algebra $\cf_\xi(\eus B)$ is commutative.
In particular, algebras $\cf_\xi(\zq)$   are natural candidates on
the r\^ole of commutative subalgebras of maximal dimension.

For semisimple $\g$, it is proved in \cite{mf3} that $\cf_\xi(\zg)$ is of maximal dimension
whenever $\xi\in\g^*\simeq\g$ is regular semisimple.
%there is an open subset $\Omega\subset\g^*$ such that   for all $\xi\in\Omega$.
A general sufficient condition for $\cf_\xi(\zq)$ to be
of maximal dimension is found by Bolsinov~\cite[Theorem~3.1]{bol2}. 
In \cite{do}, we have generalised these results and obtained a sufficient condition for $\cf_\xi(\zq)$ 
to be maximal:

\begin{thm}[see {\cite[Theorem\,3.2]{do}}]               \label{thm:cod3}
Let $\q$ be an algebraic Lie algebra.
\begin{itemize}
\item[\sf (i)] \ Suppose that $\q$ has the {\sl codim}--$2$ property
and $\eus Z(\q)$ contains algebraically independent polynomials
$f_1,\ldots,f_l$, where $l=\ind\q$, such that $\sum_{i=1}^l \deg f_i=b(\q)$.
Then $\cf_\xi(\eus Z(\q))$ is a {\bf polynomial} algebra of Krull dimension $b(\q)$ for any $\xi\in \q^*_{reg}$. 
(Hence $\cf_\xi(\zq)$ is a polynomial commutative subalgebra of maximal dimension.)

\item[\sf (ii)] \ Moreover, if $\q$ has the {\sl codim}--$3$ property, then 
$\cf_\xi(\eus Z(\q))$ is a {\bf maximal} commutative subalgebra of $\cs(\q)$.
\end{itemize}
\end{thm}

\noindent
In the rest of the paper, we consider a special class of non-reductive Lie algebras, the so-called
$\BZ_2$-{\it contractions\/} of semisimple algebras. These algebras always satisfy the conditions
stated in Theorem~\ref{thm:cod3}(i), see Theorem~\ref{thm:summary-z2} below. Therefore, one of our objectives is to study the {\sl codim}--$3$ property 
for them.

%%%%%%%%%%%%%%%%%%%%%%%%%%%
\section{Symmetric pairs, isotropy representations, and Satake diagrams}
\label{sect:sym-pairs}

\noindent
Let $\g$ be a semisimple Lie algebra and $\sigma$ an involutory automorphism of $\g$.
Let $\g_i$ denote the $(-1)^i$-eigenspace of $\sigma$.
Then  $\g_0$ is a reductive subalgebra and
$\g_1$ is an {\bf orthogonal} $\g_0$-module. We also say that $(\g,\g_0)$ is
a {\it symmetric pair} and $\g=\g_0\oplus\g_1$ is a 
$\BZ_2$-{\it grading\/} of $\g$.  Let $G$ be the adjoint group of $\g$
and $G_0$ the connected subgroup of $G$ with $\Lie G_0=\g_0$.
The representation $(G_0:\g_1)$ is  the {\it isotropy representation\/} 
of the symmetric pair $(\g,\g_0)$.

Below we introduce some notation and recall basic invariant-theoretic properties of 
the representation $(G_0:\g_1)$. The standard reference for this is \cite{kr71}.
Let $\gN$ denote the variety of nilpotent elements of $\g$. 

$(\boldsymbol{\dagger}_1)$ \ For any $v\in\g_1$ and the induced $\BZ_2$-grading  $\g_v=\g_{0,v}\oplus \g_{1,v}$, one has 
\beq   \label{eq:prop5}
\dim\g_0-\dim\g_{0,v}=\dim\g_1-\dim\g_{1,v} .
\eeq
The closure of $G_0{\cdot}v$ contains the origin if and only if $v\in\gN$; and
$G_0{\cdot}v$ is closed if and only if $v$ is semisimple. Write $G_{0,v}$ for the stabiliser of $v$ in
$G_0$, which is not necessarily connected.

$(\boldsymbol{\dagger}_2)$ \ Let $\ce\subset\g_1$  be a maximal subspace consisting of pairwise 
commuting semisimple elements. Any such subspace is called a {\it Cartan
subspace\/}. All Cartan subspaces are $G_0$-conjugate and $G_0{\cdot}\ce$ is dense in $\g_1$; 
$\dim\ce$ is called the {\it rank\/} of the $\BZ_2$-grading or pair $(\g,\g_0)$, denoted
$\rk(\g,\g_0)$.  If $h\in\ce$ is $G_0$-regular in $\g_1$, then $\g_{1,h}=\ce$ and $\g_{0,h}$ is  
the centraliser of $\ce$ in $\g_{0}$. We also write $\rr=\z(\ce)_0$ for this centraliser.

$(\boldsymbol{\dagger}_3)$ \  The algebra $\bbk[\g_1]^{G_0}$ is polynomial and 
$\dim\g_1\md G_0=\rk(\g,\g_0)$.  The quotient map $\pi: \g_1\to \g_1\md G_0$ is equidimensional,
i.e., the irreducible components of all fibres of $\pi$ are of dimension $\dim\g_1-\dim\g_1\md G_0$.
Any fibre of $\pi$ contains finitely many $G_0$-orbits and each closed $G_0$-orbit in
$\g_1$ meets $\ce$. We write $\gN(\g_1)$ for $\pi^{-1}(\pi(0))=\gN\cap\g_1$.

If $v\in\g_1$ is semisimple, % (w.l.o.g. we may assume that $h\in\ce$). 
then both $\g_v$ and $\g_{0,v}$ are reductive, and $G_0{\cdot} v$ is the unique closed orbit in $\pi^{-1}(\pi(v))$. 
By Luna's slice theorem \cite{mem33}, there is an isomorphism
\beq   \label{eq:struct-fibre}
     G_0\times_{G_{0,v}}\gN(\g_{1,v}) \isom \pi^{-1}(\pi(v)) , %\simeq  .
\eeq
which takes $(\ov{g,y})\in G_0\times_{G_{0,v}}\gN(\g_{1,v})$ to $g{\cdot}(v+y)$. This implies that 
$\g_{0,v+y}=\g_{0,v}\cap \g_{0,y}$ %=\h_{0,y}$
and $y\in \gN(\g_{1,v})$ is $G_{0,v}$-regular if and only if $v+y$ is $G_0$-regular %element
in $\pi^{-1}(\pi(v))$ and hence in $\g_1$. 
%The structure of the whole fibre can be described as follows.  
Let $\h=[\g_v,\g_v]$ and %$\h_i=\g_i\cap\h$. 
let $\z$ be the centre of $\g_v$. Then  $\g_{0,v}=\h_{0}\dotplus\z_0$ and $\g_{1,v}=\h_{1}\oplus\z_1$.
Write $H_0$ for the connected subgroup of $G_0$ with $\Lie H_0=\h_0$.
The $H_0$-orbits in $\g_{1,v}$ coincide with the orbits of the identity component of $G_{0,v}$ and,
since $\z$ consists of semisimple elements, $\gN(\g_{1,v})=\gN(\h_1)$.
Furthermore,
\beq   \label{eq:iff-codim-n}
\g_{0,v+y}=(\h_{0}\dotplus\z_0)\cap\g_{0,y}=\z_0\dotplus (\h_0\cap \g_{0,y})=
\z_0\dotplus \h_{0,y}.  
\eeq
In particular, $\g_{0,v+y}$ has the {\sl codim}--$n$ property if and only if $\h_{0,y}$ has.

Thus, the $G_0$-orbits in $\pi^{-1}(\pi(v))$ and the corresponding centralisers in $\g_0$ can be 
studied via the isotropy representation $(H_0:\h_1)$ related to the ``smaller'' symmetric pair $(\h,\h_0)$. The latter
is called a {\it reduced sub-symmetric pair\/} of $(\g,\g_0)$. 

%\indent
An explicit description of all reduced sub-symmetric pairs associated with $(\g,\g_0)$
can be given via {\it Satake diagrams}. One usually associates the Satake diagram to a real form of 
$\g$ (see e.g. \cite[Ch.\,4,\S\,4.3]{t41}). But, in view of a one-to-one correspondence between the 
real forms of $\g$ and the $\BZ_2$-gradings of $\g$, one obtains Satake diagrams for 
the symmetric pairs as well. A direct construction goes as follows.
The {\it Satake diagram\/} $\sat$ associated with $(\g,\g_0)$ is the Dynkin diagram of $\g$, where
each node is either black or white, and some pairs of white nodes are joined by a new arrow.
More precisely, choose a Cartan subspace $\ce\subset\g_1$. Let $\te$ be a $\sigma$-stable 
Cartan subalgebra of $\g$ containing $\ce$ and let $\Delta$ be the root system of $(\g,\te)$.
Since $\te$ is $\sigma$-stable, $\sigma$ acts on $\Delta$.
It is possible to choose the set of positive roots, $\Delta^+$, such that
if $\beta\in\Delta^+$ and $\beta\vert_\ce\ne 0$, then 
$\sigma(\beta)\in -\Delta^+$. Let $\Pi\subset\Delta^+$ be the corresponding set of 
simple roots. We identify the simple roots and the nodes of the Dynkin diagram.
For $\ap\in\Pi$, there are the following possibilities:

\begin{itemize}
\item \ If $\ap\vert_\ce=0$, then the root space $\g^\ap\subset \g$ belongs to $\g_0$ and $\ap$ is  black in 
$\sat$;
\item \ if $\ap\vert_\ce\ne0$ and $\sigma(\ap)=-\ap$, then $\ap$ is  white, without arrows attached;
\item \ if $\ap\vert_\ce\ne0$ and $\sigma(\ap)=-\beta\ne -\ap$, then $\beta$ is another simple root,
and the corresponding white nodes are joined by an arrow.
(In this case, we  have $\alpha\vert_{\ce}=\beta\vert_{\ce}$.)
\end{itemize}

\begin{exs}   \label{ex:diagonal-satake}
(1) Suppose that $\g=\h\dotplus\h$ and $\sigma$ is the permutation. Then $\g_0=\Delta_\h\simeq \h$ 
and $\mathsf{Sat}(\h\dotplus\h,\Delta_\h)$ is the union of two copies of the Dynkin diagram for $\h$, 
where the corresponding nodes are joined by arrows. This diagram
%Here $\mathsf{Sat}(\h\dotplus\h,\Delta_\h)$ 
has no black nodes at all.

(2) If $\ce\subset\g_1$ is a Cartan subalgebra, then $\sat$ has neither arrows nor black nodes. The corresponding symmetric pair (or, involution) is said to be of {\it maximal rank}. Recall that every simple Lie algebra has a unique, up to conjugation, involution of maximal rank.
\end{exs}
\begin{rmk}    \label{rem:indecomp}
We say that $\sat$ is  {\it connected}, if it is connected as a graph, where the new arrows
are also taken into account. For instance, if $\h$ is simple in Example~\ref{ex:diagonal-satake}(1),
then $\mathsf{Sat}(\h\dotplus\h,\Delta_\h)$ is connected. Recall that $(\g,\g_0)$ is {\it indecomposable\/} if $\g$ cannot be presented as a 
direct sum of two nonzero $\sigma$-stable ideals.
It is easily seen that $\sat$ is connected {\sl if and only if\/} $(\g,\g_0)$ is indecomposable
{\sl if and only if\/} either $\g$ is
simple, or $\h$ is simple in Example~\ref{ex:diagonal-satake}(1).
In general, $(\g,\g_0)$ is a direct sum of indecomposable symmetric pairs that correspond to the 
minimal $\sigma$-stable ideals of $\g$.
\end{rmk}

Looking at $\sat$, one immediately reads off
many properties of a symmetric pair under consideration. 
Recall that $\rr=\g_{0,h}$ for generic $h\in\ce$. Then $[\rr,\rr]$ is the semisimple subalgebra of 
$\g$ corresponding to the set of black nodes in $\sat$
and the dimension of the centre of
$\rr$ equals the number of arrows. Therefore,
\[
   \dim\ce=\rk(\g,\g_0)=\text{(number of white nodes)}-\text{(number of arrows)}.
\]
We say that $\mathsf S'$ is a {\it subdiagram of\/} 
$\sat$ if $\mathsf S'$ is obtained from $\sat$ by iterating the following 
steps: one can either remove one white node, if it does not have an arrow attached; or one 
can remove a pair of nodes connected by an arrow.
The geometric meaning of this notion is the following.

\begin{prop}{(see \cite[Prop.\,1.5]{comm-var-sym})}    \label{prop:sub-sym} 
A symmetric pair $(\h, \h_0)$ occurs as a reduced sub-symmetric pair of
$(\g,\g_0)$ if and only if\/ $\sath$ is a subdiagram of\/ $\sat$.
\end{prop}

\begin{rmk}
The connected components of\/ $\sath$ that consist only of black nodes correspond to the simple factors of $\h$ that lie entirely in $\g_0$. Therefore they do not affect $\gN(\h_1)$
and  the structure of the corresponding fibre of $\pi:\g_1\to\g_1\md G_0$.
\end{rmk}

\begin{prop}   \label{prop:ind-nilp}
Given a $\BZ_2$-grading $\g=\g_0\oplus\g_1$, let $z$ be a $G_0$-regular element of\/ $\g_1$. 
Then $\dim\g_{1,z}=\rk(\g,\g_0)$ and\/ $\ind \g_{0,z}=\rk\g-\rk(\g,\g_0)$.
\end{prop}
\begin{proof} 
Consider the induced $\BZ_2$-grading 
$\g_z=\g_{0,z}\oplus\g_{1,z}$ for  $z\in \g_1$. If $z\in\ce$  is $G_0$-regular, then $\g_{1,z}=\ce$, 
i.e., it is a (abelian) toral subalgebra of dimension  $\rk(\g,\g_0)$. Next,
$\g_{0,z}=\rr$ is reductive, with
$\ind\rr=\rk\rr=\rk\g-\rk(\g,\g_0)$, and the assertion holds in this case.

For an arbitrary $G_0$-regular $z$, we still have $\dim\g_{1,z}=\rk(\g,\g_0)$ in view of Eq.~\eqref{eq:prop5}.
Since the $G_0$-regular semisimple elements are dense in the set of  all
$G_0$-regular elements, $\g_{1,z}$ is an abelian subalgebra. By \cite[Prop.\,2.6(1)]{gnib},
the Lie bracket $\g_{0,z} \times \g_{1,z} \to  \g_{1,z}$ is trivial for the $G_0$-regular elements in 
$\g_1$. Hence $\g_z=\g_{0,z}\dotplus\g_{1,z}$ is a direct sum of Lie algebras.
Consequently, $\ind\g_{0,z}=\ind\g_z- \rk(\g,\g_0)$.
Finally, by the "Elashvili conjecture", we have $\ind\g_z=\rk\g$ for all $z\in \g$, 
see~\cite{charb-mor}.
\end{proof}

%%%%%%%%%%%%%%%
\section{Generalities on $\BZ_2$-contractions of semisimple Lie algebras}
\label{sect:new}

\noindent
For a $\BZ_2$-grading $\g=\g_0\oplus\g_1$, the semi-direct product $\ka=\g_0\ltimes\g_1$, where
$\g_1$ is an abelian ideal,
is called a $\BZ_2$-{\it contraction\/} of $\g$. This is a particular case of the general concept of 
``contractions of Lie algebras'', see \cite[Ch.\,7,\,\S\,2]{t41}.
The corresponding connected algebraic group 
is the semi-direct product $K=G_0\ltimes\g_1$.
Then $G_0$ is a Levi subgroup of $K$ and the unipotent radical of $K$, $K^u$, is  commutative. 
Of course, $K^u$ is isomorphic to $\exp(\g_1)$. The study of $\BZ_2$-contraction was initiated 
in~\cite{rims07},\cite{coadj}.
 
The vector space $\ka^*$ is isomorphic to $\g_0^*\oplus\g_1^*$. 
%and here $\g_0^*$ is a $K$-stable subspace. 
Since the Killing form is non-degenerate on $\g_i$, we obtain a fixed isomorphism of $G_0$-modules 
$\g_i\simeq\g^*_i$. Consequently, $\ka$ and $\ka^*$ are also identified as $G_0$-modules, and 
one can speak about Cartan subspaces of $\g_1^*$ and apply all invariant-theoretic results stated 
in $(\dagger_1)$-$(\dagger_3)$ to the action $(G_0:\g_1^*)$.  We write $\g_i^*$, if  we wish 
to stress that $\g_i$ is regarded as a subspace of $\ka^*$.
Upon these identifications,  
the coadjoint representation of $\ka$ is given by the following formula. If $(x_0,x_1)\in \ka$ and 
$(\xi_0,\xi_1)\in\ka^*$, then $(x_0,x_1)\star (\xi_0,\xi_1)=([x_0,\xi_0]+[x_1,\xi_1],[x_0,\xi_1])$.
It follows that $\g_0^*$ is a $K$-stable subspace of $\ka^*$ and $\ka^*/\g_0^*$ is a $K$-module with
trivial action of $K^u$. That is, $\ka^*/\g_0^*$ is identified with the $G_0$-module $\g_1^*$.

We summarise below some fundamental results on the coadjoint representation of $\ka$.
% is obtained in~\cite{coadj}.
\begin{thm} \label{thm:summary-z2}
For any $\BZ_2$-contraction $\ka$ of a semisimple Lie algebra $\g$, we have
\begin{itemize}
\item[\sf (i)] \ $\ind\ka=\ind\g=\rk\g$ \cite[Prop.\,9.3]{rims07}. Therefore $b(\ka)=b(\g)$.
\item[\sf (ii)] \ $\ka$ has the {\sl codim}--$2$ property  \cite[Theorem\,3.3]{coadj}.
\item[\sf (iii)] \ The algebra $\bbk[\ka^*]^K$ is polynomial, of Krull dimension $l=\rk\g$. 
If $f_1,\dots,f_l$ are alge\-bra\-i\-cally independent homogeneous generators of\/
$\bbk[\ka^*]^K$, then $\sum_{i=1}^l\deg f_i=b(\ka)$~\cite{coadj,kys-1param,kys-4cases}.
\end{itemize}
\end{thm}

\begin{rmk}     \label{rmk:pro-klassification}
The first assertion in (iii) is achieved via case-by-case considerations. Namely,
results of~\cite{coadj,kys-1param} together cover all but four cases related to simple algebras of type
$\GR{E}{n}$. The remaining involutions %are 
will be handled in \cite{kys-4cases}.
In each case, a certain set of homogeneous generators of $\bbk[\ka^*]^K$
can be constructed, and the equality $\sum_{i=1}^l\deg f_i=b(\ka)$ arises a posteriori.
However, once one knows somehow that the algebra $\bbk[\ka^*]^K$ is polynomial, there is also a 
conceptual way to establish the equality for the sum of degrees:
\\ \indent
By \cite[Theorem\,2.2]{jf-TG}, which generalises a sum rule obtained in  \cite[Theorem\,1.1]{fonya-et}, 
 if a Lie algebra $\q$ is unimodular, %%, has the 
$\bbk[\q^*]^Q$ is  polynomial, and the fundamental semi-invariant ${\bf p}_\q$ of $\q$ is
an invariant, then the sum of degrees of generators of $\bbk[\q^*]^Q$ equals $b(\q){-}\deg{\bf p}_\q$.
One easily proves that all $\BZ_2$-contractions are unimodular, and since they have the 
{\sl codim}--$2$ property, ${\bf p}_\ka=1$ for all of them. 
\end{rmk}

\subsection{$\BZ_2$-contractions associated with $\gN$-regular symmetric pairs} 
\label{sect:n-reg}
A symmetric pair $(\g,\g_0)$ 
is called $\gN$-{\it regular\/} if $\g_1$ contains a regular nilpotent element of $\g$.
By a result of Antonyan~\cite{leva}, a symmetric pair 
is  $\gN$-regular if and only if $\g_1$ contains a regular
semisimple element of $\g$. Therefore, $\gN$-regularity is equivalent to that $\sat$ contains no 
black nodes. For the $\gN$-regular symmetric pairs, $m:=\rk\g-\rk(\g,\g_0)$ is equal to the number of 
arrows in $\sat$.
The list of all indecomposable $\gN$-regular symmetric pairs includes
the symmetric pairs of maximal rank (with $m=0$) and also the following pairs:
\begin{enumerate}
\item $(\mathfrak{sl}_{n+k}, \mathfrak{sl}_n\dotplus\mathfrak{sl}_k\dotplus \te_1)$, 
$|n-k|\le 1$, \ $m=\min\{n,k\}$;
\item $(\mathfrak{so}_{2n+2}, \mathfrak{so}_n\dotplus\mathfrak{so}_{n+2})$, \ $m=1$;
\item $(\eus {E}_{6}, \mathfrak{sl}_6\dotplus \tri)$, \ $m=2$;
\item $(\h\dotplus\h,\Delta_\h)$, where $\h$ is simple, \ $m=\rk\h$.
\end{enumerate}
%\\
The following result is a straightforward consequence of the theory developed in \cite{coadj}.

\begin{thm}      \label{thm:N-reg}
Suppose that $\g=\g_0\oplus\g_1$ is $\gN$-regular. Then
$\bbk[\ka^*]^{K^u}=\bbk[\ka^*]^{\g_1}$ is a maximal commutative subalgebra 
of\/ $\bbk[\ka^*]$.
\end{thm}
\begin{proof}
For the $\gN$-regular $\BZ_2$-gradings, $\bbk[\ka^*]^{K^u}$ is a polynomial algebra of Krull 
dimension $\dim\g_1+\rk\g-\rk(\g,\g_0)=b(\g)$.
More precisely, let $e_1,\dots,e_n$ be a basis for $\g_1$. We regard the $e_i$'s as linear function on
$\g_1^*$ and hence on $\ka^*$.
Then $\bbk[\ka^*]^{K^u}$ is freely generated by $e_1,\dots,e_n$, $\hat F_1,\dots, \hat F_m$, 
%\[  \bbk[\ka^*]^{K^u}= \bbk[e_1,\dots,e_n, \hat F_1,\dots, \hat F_m] \ ,%\]
where $m=\rk\g-\rk(\g,\g_0)$ and $\hat F_1,\dots,\hat F_m$ are explicitly described polynomials that 
are even $K$-invariant \cite[Theorem\, 5.2]{coadj}. Since $\g_1$ is an Abelian ideal in $\ka$ and the 
$\hat F_j$'s belong to the centre of the Poisson algebra $\bbk[\ka^*]$,
the algebra $\bbk[\ka^*]^{K^u}$ is commutative.
On the other hand, if $\ca\supset \bbk[\ka^*]^{K^u}$, then $\ca$ contains the whole
space $\g_1$. Hence the commutativity of $\ca$ implies that
$\ca\subset \bbk[\ka^*]^{\g_1}$.
\end{proof}

\noindent  Since all elements of $\cs(\ka)^K$ can naturally be lifted to the centre of $\eus U(\ka)$ and there is no problem with lifting elements of degree $1$, 
the above description of free generators shows that 
$\cs(\ka)^{\g_1}=\bbk[\ka^*]^{\g_1}$ can be lifted to the enveloping algebra $\eus U(\ka)$; that is,
there exists a commutative subalgebra $\tilde\ca\subset \eus U(\ka)$ such that $\gr(\tilde\ca)=
\cs(\ka)^{\g_1}$. In particular, 
if $m=0$, i.e., $(\g,\g_0)$ is of maximal rank, then just $\cs(\g_1)=\bbk[\g_1^*]$ is a maximal 
commutative subalgebra of $\eus U(\ka)$. 

\begin{rmk}
The algebra $\bbk[\ka^*]^{\g_1}$ can be considered for any $\BZ_2$-contraction. However, we can
prove  that
\\
\centerline{the algebra $\bbk[\ka^*]^{\g_1}$ is commutative $\Longleftrightarrow$ $(\g,\g_0)$ is $\gN$-regular.
}
Recall that $\rr=\z(\ce)_0=\g_{0,\xi}$ for a Cartan subspace $\ce\in\g_1\simeq\g_1^*$ and 
generic $\xi\in\ce$. Then 
\begin{multline*}
\trdeg \bbk[\ka^*]^{\g_1}=\dim\ka-\max_{\xi\in\ka^*} \dim K^u{\cdot}\xi=
\dim\ka-\dim\g_1+\min_{\xi\in\g_1^*}\dim\g_{1,\xi}=
\\
=\dim\ka-\dim\g_0+\min_{\xi\in\g_1^*}\dim\g_{0,\xi}=\dim\g_1+\dim \rr .
\end{multline*}
On the other hand, Ra\"\i s' formula for the index of semi-direct product \cite{rais} 
%$\ka=\g_0\ltimes\g_1$ 
shows that $\ind\ka=\dim\g_1-\dim\g_0+\dim\rr+\ind\rr$ and therefore
\[
  b(\ka)=(\dim\ka+\ind\ka)/2=\dim\g_1+b(\rr).
\] 
It follows that $\trdeg \bbk[\ka^*]^{\g_1}\ge b(\ka)$ and the equality occurs if and
only if $\rr$ is toral. The latter is equivalent to that $\ce$ contains regular semisimple elements of $\g$, i.e.,
$(\g,\g_0)$ is $\gN$-regular.
\end{rmk}

For involutions (symmetric pairs) of maximal rank, $\dim\g_1=b(\g)$. Therefore, the 
commutative Lie subalgebra $\g_1$ is a {\it commutative polarisation\/} of $\ka$~\cite[Sect.\,5]{fonya00}.
Conversely, using \cite[Prop.\,20]{fonya00}, one can prove that if $\ka$ admits a commutative polarisation, then $(\g,\g_0)$ is
of maximal rank.

%%%%%%%%%%%%%%%%%%%%%%%%%%%
\section{$\BZ_2$-contractions with codimension--3 property}
\label{sect:codim3}

\noindent
The {\sl codim}--$3$ property does not hold for all $\BZ_2$-contractions.
In~\cite[Example\,4.1]{do}, we noticed  that, for the involutions of maximal rank and
$\xi\in\ka^*_{reg}\cap\g_1^*$, the commutative subalgebras $\cf_\xi(\eus Z(\ka))$ fail to be maximal, and 
thereby $\ka$ does not have the {\sl codim}--$3$ property. 
In this section, we obtain a characterisation of $\BZ_2$-contractions with the {\sl codim}--$3$ property.

\begin{df}
A node of the Satake diagram $\sat$ is said to be {\it trivial}, if it is white, 
does not have an arrow attached, and all adjacent nodes are also white. 
\end{df}

\begin{thm}   \label{thm:main}
A $\BZ_2$-contraction $\ka=\g_0\ltimes\g_1$ of a semisimple algebra $\g$
%Let $\g$ be semisimple. Then a $\BZ_2$-contraction $\ka=\g_0\ltimes\g_1$ 
has the {\sl codim}--$3$ property if and only if\/ $\sat$ has no trivial nodes.
\end{thm}

\begin{rmk}    \label{bad-rank-1} A description of the reduced sub-symmetric pairs via $\sat$ 
(see Prop.~\ref{prop:sub-sym}) shows that this theorem can be restated as follows. 
A $\BZ_2$-contraction $\ka=\g_0\ltimes\g_1$ has the {\sl codim}--$3$ property if and only if\/ 
there is no reduced sub-symmetric pairs of rank $1$ associated with
$(\g,\g_0)$ of the form %does not contain a direct sum of two symmetric pairs 
$(\h,\h_0)=(\mathfrak{sl}_2,\mathfrak{so}_2){\dotplus}(\el, \el)$, where $\el$ is a semisimple 
Lie algebra. (That is, $\tri$ is the only simple ideal of $\h$ with non-trivial restriction of $\sigma$.)
%\\  \indent
It is also easily verified that the $\BZ_2$-contraction $\ka$ 
arising from the symmetric pair $(\tri,\mathfrak{so}_2)$ does not have the {\sl codim}--$3$ property.
\end{rmk}

Before dwelling on the proof, we provide an explicit description of the corresponding
symmetric pairs. If $(\g,\g_0)$ is not indecomposable, then $\ka$ is the direct sum of the 
$\BZ_2$-contractions corresponding to the minimal $\sigma$-stable ideals of $\g$ (cf. 
Remark~\ref{rem:indecomp}). Therefore, it suffices to point out the admissible 
indecomposable symmetric pairs.

\begin{cl}  \label{cl:tablitsa}
Suppose that $(\g,\g_0)$ is indecomposable and\/ $\ka=\g_0\ltimes\g_1$ has the {\sl codim}--$3$ 
property. Then either $\g=\h\dotplus\h$, where $\h$ is simple and  $\g_0=\Delta_\h$, or $\g$ is simple and $(\g,\g_0)$ occurs in Table~\ref{table:all}.
\end{cl}

\begin{table}[htb]  
\begin{center}
\begin{tabular}{clccc}   
  &  $(\g,\g_0)$ &  $\sat$  &  $\text{rk}(\g,\g_0)$ & $\rr=\z(\ce)_0$ \\  \hline
1) & $(\sln, \mathfrak{sl}_k\dotplus\mathfrak{sl}_{n-k}\dotplus \te_1)$, %\ 0{<}k {<} n{-}k$ 
&  
\begin{picture}(155,38)(0,7)
\setlength{\unitlength}{0.016in}
\multiput(10,8)(30,0){2}{\circle{5}}
\multiput(55,8)(25,0){2}{\circle*{5}}
\multiput(95,8)(30,0){2}{\circle{5}}
\multiput(43,8)(40,0){2}{\line(1,0){9}}
\multiput(13,8)(45,0){2}{\line(1,0){4}}
\multiput(33,8)(40,0){2}{\line(1,0){4}}
\multiput(98,8)(20,0){2}{\line(1,0){4}}
\multiput(19.5,5)(42,0){3}{$\cdots$}
\multiput(19.5,13)(84,0){2}{$\cdots$}
{\color{blue}\put(67.5,14){\oval(115,30)[t]}
\put(67.5,14){\oval(55,20)[t]}
\multiput(10,14)(30,0){2}{\vector(0,-1){3}}
\multiput(95,14)(30,0){2}{\vector(0,-1){3}}}
\end{picture}
& $k$  & $\mathfrak{sl}_{n-2k}\dotplus\te_k$ \\
 & $0{<}k {<} n{-}k$ & & & 
\\
2) & $(\mathfrak{sl}_{2n},\spn)$, $n\ge 2$ & 
\begin{picture}(130,20)(-10,5)
\put(10,8){\circle*{6}}    % first 
\put(30,8){\circle{6}}     %  second 
\put(13,8){\line(1,0){14}}
\multiput(70,8)(40,0){2}{\circle*{6}} 
\put(90,8){\circle{6}}
\multiput(73,8)(20,0){2}{\line(1,0){14}}
\multiput(33,8)(29,0){2}{\line(1,0){5}}   % two short lines
\put(42,5){$\cdots$}
\end{picture}
& $n-1$  & $(\tri)^n$ 
\\
3) & $(\mathfrak{so}_{4n+2}, \mathfrak{gl}_{2n+1}),  n{\ge} 2$ & 
\begin{picture}(150,30)(0,5)
\put(10,8){\circle*{6}}    % first 
\put(30,8){\circle{6}}     %  second 
\put(13,8){\line(1,0){14}}
\multiput(70,8)(40,0){2}{\circle*{6}} 
\put(90,8){\circle{6}}
\multiput(73,8)(20,0){2}{\line(1,0){14}}
\multiput(130,-2)(0,20){2}{\circle{6}}
\put(113,10){\line(2,1){13}}
\put(113,6){\line(2,-1){13}}
\multiput(33,8)(29,0){2}{\line(1,0){5}}   % two short lines
\put(42,5){$\cdots$}
{\color{blue}\put(137,8){\oval(20,20)[r]}}
{\color{blue}\put(133,18){\vector(-1,0){2}}}
{\color{blue}\put(130,-2){\vector(-1,0){2}}}
\end{picture}   & $n$  & $(\tri)^n\dotplus\te_1$ \\
4) & $(\son,\mathfrak{so}_{n-1}), \ n\ge 5$ &  
\begin{picture}(150,15)(0,5)
\put(10,8){\circle{6}}    % first 
\multiput(30,8)(20,0){2}{\circle*{6}}
\multiput(13,8)(20,0){3}{\line(1,0){14}}
\put(72,5){$\cdots$}
\end{picture} 
& $1$ & $\mathfrak{so}_{n-2}$
\\
5) & $(\spn, \mathfrak{sp}_{2k}\dotplus\mathfrak{sp}_{2n-2k})$,  & %\ k \le n{-}k$ & 
\begin{picture}(165,25)(0,5)
\multiput(30,8)(60,0){2}{\circle{6}}
\multiput(70,8)(40,0){3}{\circle*{6}}
\multiput(10,8)(160,0){2}{\circle*{6}}
\multiput(152.5,7)(0,2){2}{\line(1,0){15}}
\multiput(73,8)(20,0){2}{\line(1,0){14}}
\multiput(33.1,8)(29,0){2}{\line(1,0){5}}
\multiput(113,8)(29,0){2}{\line(1,0){5}}
\put(42,5){$\cdots$}
\put(122,5){$\cdots$}
\put(13,8){\line(1,0){14}}
%\put(8,2){$\underbrace%
%{\mbox{\hspace{86\unitlength}}}_{2m}$}
\put(155,5){$<$} 
\end{picture}   & $k$  & $(\tri)^k{\dotplus} \mathfrak{sp}_{2n-4k}$
\\
 & $1\le k \le n{-}k$ & & &
\\
6) & $(\eus E_{6},\eus {F}_{4})$ & 
\begin{picture}(90,25)(25,10) 
\setlength{\unitlength}{0.016in} 
\multiput(38,18)(15,0){4}{\line(1,0){9}}
\put(65,3){\circle*{5}}
\multiput(50,18)(15,0){3}{\circle*{5}}
\multiput(35,18)(60,0){2}{\circle{5}}
\put(65,6){\line(0,1){9}}
%\put(95,18){\circle{5}}
\end{picture}  & $2$ & $\mathfrak{so}_8$ 
\\
7) & $(\eus E_{6},\mathfrak{so}_{10}\dotplus\te_1)$ &  
\begin{picture}(90,45)(25,10) 
\setlength{\unitlength}{0.016in} 
\multiput(38,18)(15,0){4}{\line(1,0){9}}
\put(65,3){\circle{5}}
\multiput(50,18)(15,0){3}{\circle*{5}}
\multiput(35,18)(60,0){2}{\circle{5}}
\put(65,6){\line(0,1){9}}
{\color{blue}\put(65,23){\oval(60,20)[t]}
\multiput(35,23)(60,0){2}{\vector(0,-1){2}}}
\end{picture}
& $2$  &  $\mathfrak{sl}_4\dotplus\te_1$
\\
8) & $(\eus {F}_{4}, \mathfrak{so}_{9})$ & 
\begin{picture}(90,20)(40,5)  
\setlength{\unitlength}{0.016in} 
\put(50,8){\circle{5}}
\multiput(70,8)(20,0){3}{\circle*{5}}
\multiput(53,8)(40,0){2}{\line(1,0){14}}
\multiput(72.5,6.6)(0,2.6){2}{\line(1,0){14.8}}
\put(76,5.2){$<$} 
\end{picture}
& $1$ & $\mathfrak{so}_7$ 
\\ \hline
\end{tabular} \vskip1.2ex
\caption{The symmetric pairs with $\g$ simple and the {\sl codim}--$3$ property for $\ka$.}%=\g_0\ltimes\g_1$}
\label{table:all}
\end{center}
\end{table}
{\sl Remarks on Table~\ref{table:all}.}
(i)  The number of black nodes in item~1) equals $n-1-2k$ and the number of arrows equals
$k$. If $k+1=n-k$, then $\sat$ has no black nodes at all.
\\ \indent
(ii) The right-hand end of $\sat$ in item 4) depends on the parity of $n$ (type {\bf B} or
{\bf D}).

\begin{proof}[Proof of Theorem~\ref{thm:main}]
We begin with describing certain $K$-regular elements of $\ka^*\simeq \g_0^*\oplus\g_1^*$.
Consider the mappings
\[
    \ka^*  \stackrel{\psi}{\longrightarrow}  \g_1^*  \stackrel{\pi}{\longrightarrow} \g_1^*\md G_0 \ ,
    %=\bigsqcup_i (\g_1^*\md G_0)_i \ , 
\]
where $\psi$ is the projection with kernel $\g_0^*$ and $\pi$ is the quotient morphism. 
Recall that $\g_0^*$ is a $K$-submodule of
$\ka^*$, hence $\psi$ is a surjective homomorphism of $K$-modules (the unipotent radical
$K^u$ acts trivially on $\g_1^*$).
Let $\eta=(\ap,\beta)\in \ka^*$ be an arbitrary point, 
where $\ap\in\g_0^*$ and $\beta\in\g_1^*$. 
Write $\g_{0,\beta}$ for the stabiliser of $\beta$ in $\g_0$.
Then $\g_1\star\beta=\Ann(\g_{0,\beta})\subset\g_0^*$ and 
therefore $\g_0^*/(\g_1\star\beta) \simeq \g_{0,\beta}^*$.
Using the last isomorphism, we let $\hat\ap$ denote the image of $\ap$ in $\g_{0,\beta}^*$.
By  \cite[Prop.\,5.5]{rims07}, %we have
\begin{equation}  \label{eq:dim_stab}
  \dim \ka_{\eta}=\codim_{\g_1^*}(G_0{\cdot}\beta)+\dim (\g_{0,\beta})_{\hat\ap} \ ,
\end{equation}
where the last summand refers to the stabiliser of $\hat\ap$ with respect to the
coadjoint representation of $\g_{0,\beta}$.
%For any $\beta\in\g_1^*$,  
Since $\psi^{-1}(G_0{\cdot}\beta)=\g_0^*\times G_0{\cdot}\beta$ is $K$-stable,
it follows from Eq.~\eqref{eq:dim_stab} that
\beq  \label{eq:min-codim}
  \min \bigl(\codim_{\ka^*} \text{\{$K$-orbits in $\psi^{-1}(G_0{\cdot}\beta)$\}}\bigr)=\codim_{\g_1^*}(G_0{\cdot}\beta)+
  \ind(\g_{0,\beta}) . 
\eeq
If $\beta\in(\g_1^*)_{reg}$, then $\ind(\g_{0,\beta})=\rk\g-\rk(\g,\g_0)$ (see Proposition~\ref{prop:ind-nilp})
and $\codim_{\g_1^*}(G_0{\cdot}\beta)=\rk(\g,\g_0)$. Consequently, 
\beq   \label{eq:fibre-soderzhit}
\text{if $\beta\in(\g_1^*)_{reg}$, \ then $\psi^{-1}(G_0{\cdot}\beta)=\g_0^*\times G_0{\cdot}\beta$  \ contains $K$-regular elements.}
\eeq
Consider the Luna stratification of the quotient variety $\g_1^*\md G_0$, see \cite[III.2]{mem33}.
By definition, $\bar\xi,\bar\xi'\in \g_1^*\md G_0$ belong to the same stratum, if 
the closed $G_0$-orbits in $\pi^{-1}(\bar\xi)$ and $\pi^{-1}(\bar\xi')$
are isomorphic as $G_0$-varieties. Each stratum is locally closed, and there are finitely many of them. (An exposition of Luna's theory can also be found in~\cite{slo}.)
Write $\bar\Omega_i$ for the union of all strata of codimension $i$.
In particular, $\bar\Omega_0$ is the unique open stratum.

Set $\Omega_i=\pi^{-1}(\bar\Omega_i)$ 
and $\Xi_i=\psi^{-1}(\Omega_i)=\g_0^*\times\Omega_i$.
 Since both $\pi$ and $\psi$ are equidimensional,
$\codim_{\ka^*}\Xi_i=\codim_{\g_1^*}\Omega_i=i$. Therefore,
$\Xi_0\cup\Xi_1\cup\Xi_2$ has the complement of codimension $\ge 3$ in $\ka^*$, 
$\Omega_0\cup\Omega_1\cup\Omega_2$ has the complement of codimension $\ge 3$ in $\g_1^*$,
and we may not care about the strata of codimension $\ge 3$. Our ultimate goal is to characterise
the symmetric pairs such that  $(\Xi_0\cup\Xi_1\cup\Xi_2)\cap \ka^*_{reg}$ still
has the complement of codimension $\ge 3$ %in is still very big
in $\ka^*$. More precisely, we are going to find out whether 
$(\Xi_i)_{sg}:=\Xi_i \setminus (\Xi_i\cap\ka^*_{reg})$ is of codimension $\ge (3-i)$ in $\Xi_i$.
It appears to be that for $i=0,2$, this condition is satisfied for all $\BZ_2$-contractions, and 
non-trivial constraints occur only  for $i=1$.
  
{\color{MIXT}$(\Xi_0)$-case.} \ 
If $\bar\xi\in\bar\Omega_0$, then $\pi^{-1}(\bar\xi)=G_0{\cdot}\xi$ (a sole closed and $G_0$-regular
orbit!). Here $\g_{0,\xi}$ is reductive and
\[
   \Xi_0=\bigsqcup_{\bar\xi\in \bar\Omega_0} (\g_0^*\times G_0{\cdot}\xi) .
\]
It follows from Eq.~\eqref{eq:dim_stab} that $(\ap,\xi)\in \ka^*_{reg}$ if and only if $\hat\ap$ is 
$\g_{0,\xi}$-regular. Since $\g_{o,\xi}$ has the {\sl codim}--$3$ property (see 
Example~\ref{ex:codim3-red}), we conclude that $(\Xi_0)_{sg}$
%$\Xi_0\setminus (\Xi_0\cap\ka^*_{reg})$ 
is of codimension $\ge 3$.

{\color{MIXT}$(\Xi_1)$-case.} \ If $\bar\xi\in\bar\Omega_1$, % \cup\bar\Omega_2$, 
then $\pi^{-1}(\bar\xi)$ is not a sole orbit.  Below, we use the notation that 
\begin{itemize}
\item[\bf --]  \  $\xi\in\pi^{-1}(\bar\xi)$ is {\bf semisimple} and, without loss of generality, we assume that $\xi\in\ce$;
%hence $G_0{\cdot}\xi$ is the unique  closed orbit in $\pi^{-1}(\bar\xi)$.
\item[\bf --]  \ $\zeta\in\pi^{-1}(\bar\xi)$ is $G_0$-{\bf regular} and hence $G_0{\cdot}\zeta$ is open 
in $\pi^{-1}(\bar\xi)$;
\item[\bf --]  \ $\pi^{-1}(\bar\xi)_{sg}$ is the complement of set of $G_0$-regular elements of $\pi^{-1}(\bar\xi)$.
\end{itemize}
Let $\bar\Omega_1^{(j)}$ be a Luna stratum of codimension~$1$ and 
$\Omega_1^{(j)}:=\pi^{-1}(\bar\Omega_1^{(j)})$, $\Xi_1^{(j)}:=\psi^{-1}(\Omega_1^{(j)})$ the 
corresponding strata in $\g_1^*$ and $\ka^*$. Then
\begin{multline*}
   \Xi_1^{(j)}=\g_0^*\times\Omega_1^{(j)}=\bigsqcup_{\bar\xi\in\bar\Omega_1^{(j)}}(\g_0^*\times \pi^{-1}(\bar\xi))= \\
  \bigl(\bigsqcup_{\zeta}\ (\g_0^*\times G_0{\cdot}\zeta)\bigr)\cup 
  \bigl(\bigsqcup_{\bar\xi\in\bar\Omega_1^{(j)}}(\g_0^*\times \pi^{-1}(\bar\xi)_{sg})\bigr)=:
  \mathfrak Y^{(j)}\cup\mathfrak Z^{(j)} ,
\end{multline*}
where $\zeta$ ranges over the set of representative of all $G_0$-regular orbits in 
$\Omega_1^{(j)}$. 

The required information on $\mathfrak Y^{(j)}$ and $\mathfrak Z^{(j)}$ will be extracted from the 
coadjoint representation of $\g_{0,\zeta}$ and the Satake diagram associated with the closed orbits 
in $\Omega_1^{(j)}$, respectively. 
For $\xi\in\Omega_1^{(j)}\cap\ce$, the corresponding reduced sub-symmetric pair is of rank~$1$. 
As in Section~\ref{sect:sym-pairs}, we consider $\h=[\g_\xi,\g_\xi]=\h_0\oplus\h_1$ and the action
$(H_0:\gN(\h_1))$. 

\textbullet \ By Eq.~\eqref{eq:dim_stab} and \eqref{eq:fibre-soderzhit}, $\mathfrak Y^{(j)}$ contains 
$K$-regular elements and the dimension of their complement is determined by the coadjoint 
representation of $\g_{0,\zeta}$. Namely, if $\g_{0,\zeta}$ has the {\sl codim}--$n$ property, then
$\codim_{\ka^*} (\mathfrak Y^{(j)}\setminus (\mathfrak Y^{(j)}\cap\ka^*_{reg}))=n+1$. Hence we
need the {\sl codim}--$2$ property for $\g_{0,\zeta}$, i.e., for stabilisers of $G_0$-regular elements
in $\Omega_1^{(j)}$.

\textbullet \ For $\mathfrak Z^{(j)}$, we have $\codim_{\ka^*}\mathfrak Z^{(j)}=
1+\codim_{\pi^{-1}(\bar\xi)}\pi^{-1}(\bar\xi)_{sg}=
1+\codim_{\gN(\h_1)} \gN(\h_1)_{sg}$. Hence $\mathfrak Z^{(j)}$ is irrelevant for the 
{\sl codim}--$3$ property whenever $\codim_{\gN(\h_1)} \gN(\h_1)_{sg}\ge 2$. If 
$\codim_{\gN(\h_1)} \gN(\h_1)_{sg}=1$, then a more accurate analysis of $\mathfrak Z^{(j)}$ is
needed. 

Because $\rk(\h,\h_0)=1$ and $\sath$ is a sub-diagram of $\sat$ 
(Proposition~\ref{prop:sub-sym}), $\sath$ contains all black nodes from $\sat$ and 
either a unique white node or a unique pair of white nodes joined by an arrow.
We consider all the possibilities in turn.

{\sf (I)} \  {\sl %Assume that\/ 
$\sath$ contains a unique white node and this node is trivial in\/ $\sat$.} 
\\
In other words,  $\sath$ is a {\bf disjoint} union of 
one white node and the subdiagram of all black nodes. In this case, $\h=\tri\dotplus [\rr,\rr]$ and 
$[\rr,\rr]\subset \h_0$.  That is, the $\BZ_2$-grading of $\h$ is determined by the unique non-trivial
$\BZ_2$-grading of $\tri$.
Here $\dim\h_1=2$ and $\gN(\h_1)$ is the union of two lines in $\h_1$ (the ``coordinate cross'').
Using \eqref{eq:struct-fibre}, we obtain that 
$G_0{\cdot}\xi$ is of codimension $1$ in $\pi^{-1}(\bar\xi)$ (and $\pi^{-1}(\bar\xi)$ also contains
two $G_0$-orbits of regular elements). Therefore, the union of all closed $G_0$-orbits in 
$\Omega_1^{(j)}$ yields a subvariety $Z^{(j)}$, of codimension~1 in 
$\Omega_1^{(j)}$, hence of codimension~2 in $\g_1^*$. Let us prove that 
$\mathfrak Z^{(j)}=\psi^{-1}(Z^{(j)})$ does not contain $K$-regular elements.
Indeed, in this case $\g_{0,\xi}$ is reductive and $\dim\g_{0,\xi}=\dim\rr+1$. Hence
$\g_{0,\xi}=\rr\dotplus\te_1$. It follows that
%, where $\rr$ is the centraliser of $\ce$ in $\g_0$. Thus, 
$\ind(\g_{0,\xi})=\ind\rr+1=\rk\g-\rk(\g,\g_0)+1$. Now, using Eq.~\eqref{eq:min-codim} and the fact 
that $\codim_{\g_1^*}(G_0{\cdot}\xi)=\rk(\g,\g_0)+1$, we obtain that 
$\dim \ka_\eta\ge \rk\g+2=\ind\ka+2$ for all $\eta\in\mathfrak Z^{(j)}$.
Thus, here $\ka$ does {\bf not} have the {\sl codim}--$3$ property.

{\sf (II)}\ \ {\sl %Assume that\/ 
$\sath$ has a unique white node which is adjacent to
a black node}. 
\\
To realise the structure of $\pi^{-1}(\bar\xi)\simeq G_0\times_{G_{0,\xi}}\gN(\h_1)$, we may only 
consider the connected component of $\sath$ that contains the white node.
That is, we  look at $\gN(\h_1)$ for the $\BZ_2$-gradings of rank one of 
{\bf simple} Lie algebras such that the Satake diagram has no arrows. The corresponding list consists 
of the following symmetric pairs $(\h,\h_0)$:
\begin{center}
%$\boldsymbol I_1$: \ $(\mathfrak{sl}_4, \mathfrak{sp}_{4})$; \ 
$\boldsymbol I_1$: \ $(\son, \mathfrak{so}_{n-1})$, $n\ge 5$; \ 
$\boldsymbol I_2$: \ $(\spn,\mathfrak{sp}_{2n-2}\dotplus\mathfrak{sp}_2)$, $n\ge 2$; \ 
$\boldsymbol I_3$: \ $(\eus {F}_{4}, \mathfrak{so}_{9})$.
\end{center}
[Note that the case $(\boldsymbol I_1,\,n=5)$ coincides with $(\boldsymbol I_2,\, n=2)$, and 
$(\boldsymbol I_1,\,n=6)$ is also equal to $(\mathfrak{sl}_4, \mathfrak{sp}_{4})$.]
For all three cases, we have $\codim_{\gN(\h_1)} \gN(\h_1)_{sg} \ge 2$.
Hence $\mathfrak Z^{(j)}$ is irrelevant for the {\sl codim}--$3$ property for $\ka$ and,
as explained above, we only have to verify
the {\sl codim}--$2$ property for $\g_{0,\zeta}$.
Without loss of generality, we may assume that
$\zeta=\xi+y$, where $y\in\gN(\h_1)$ is $H_0$-regular. Then $\g_{0,\zeta}=\z_0\dotplus
\h_{0,y}$, see Eq.~\eqref{eq:iff-codim-n}, and it suffices to check the {\sl codim}--$2$ property for 
$\h_{0,y}$.
%the stabilisers of $H_0$-regular nilpotent elements of $\h_1$ 
Because $\rk(\h,\h_0)=1$, we have $\dim\h_{1,y}=1$. 
Hence %$\h_{1,y}=\langle y\rangle$ and 
$\h_y=\h_{0,y}\dotplus \langle y\rangle$, and we can work with either $\h_y$ or $\h_{0,y}$
in the above cases $(\boldsymbol I_1$-$\boldsymbol I_3)$.

For $\h=\sln$ or $\spn$,  {\bf all} the centralisers $\h_v$
have the {\sl codim}--$2$ property~\cite[Sect.\,3]{ppy},
which covers cases $(\boldsymbol I_1,\,n\le 6)$ and $\boldsymbol I_2$.
For $(\boldsymbol I_1,\,n\ge 7)$, an explicit description of centralisers shows that
$\h_{0,y}$ is a $\BZ_2$-contraction  of  $\mathfrak{so}_{n-2}$, and all 
$\BZ_2$-contractions have the {\sl codim}--$2$ property.
For $\boldsymbol I_3$, we have $\h_{0,y}=\eus G_2\ltimes \bbk^7$, which is a contraction of
$\mathfrak{so}_7$, %\sfr(\varpi_1)$ 
and the verification is straightforward.

Thus, the strata occurring in part~{\sf (II)} provide no obstacles for the {\sl codim}--$3$ property.

{\sf (III)}\ \ {\sl  
$\sath$ has a unique pair of nodes joined by an arrow}.
\ 
There are three possibilities for the connected component of $\sath$ containing these two 
white nodes.

({\sf III}-{\it a})  {\sl There are black nodes between the white ones}. The corresponding connected 
component of $\sath$ looks like item 1) in Table~\ref{table:all} with $k=1$ and $n\ge 4$. Here again
$\codim_{\gN(\h_1)} \gN(\h_1)_{sg} \ge 2$ and the  argument goes through as in part {\sf (II)} 
of the proof.  An essential point is that  $\h_{0,y}$ has the {\sl codim}--$2$ property, because it is a 
$\BZ_2$-contraction of $\mathfrak{gl}_{n-2}$.

({\sf III}-{\it b}) {\sl The two white nodes are {\bf not} adjacent in the Dynkin diagram of $\g$, and there is 
no black nodes between them}. Such a sub-diagram occurs only for the symmetric pair
$(\mathfrak{sl}_{2n+1}, \sln\dotplus\slno\dotplus\te_1)$ ($n\ge 2$) and $\sath$ is just
\begin{picture}(45,20)(-5,5)
\setlength{\unitlength}{0.017in}
\multiput(5,5)(20,0){2}{\circle{5}}
{\color{blue}\put(15,10){\oval(20,10)[t]}
\multiput(5,10)(20,0){2}{\vector(0,-1){2}}
}
\end{picture}. Here $\h=\tri\dotplus\tri$ and $\h_0$ is the diagonal in $\h$
(cf. Example~\ref{ex:diagonal-satake}). In this case, we have $\dim\gN(\h_1)=2$ and 
$\gN(\h_1)_{sg}=\{0\}$. Hence $\codim_{\gN(\h_1)} \gN(\h_1)_{sg} = 2$.
If $y\in \gN(\h_1)\setminus \{0\}$, then $\h_{0,y}$ is $1$-dimensional and abelian. Therefore,
$\h_{0,y}$ has the {\sl codim}--$2$ property.

({\sf III}-{\it c}) {\sl The two white nodes are adjacent in the Dynkin diagram}, i.e., $\sath$ contains a 
connected component
\begin{picture}(45,20)(-5,5)
\setlength{\unitlength}{0.017in}
\multiput(5,5)(20,0){2}{\circle{5}}
\put(8,5){\line(1,0){14}}
{\color{blue}\put(15,10){\oval(20,10)[t]}
\multiput(5,10)(20,0){2}{\vector(0,-1){2}}
}
\end{picture}.
Again, this means  that $(\g,\g_0)=(\mathfrak{sl}_{2n+1}, \sln\dotplus\slno\dotplus\te_1)$, 
with $\rk(\g,\g_0)=n$, and there is no black nodes at all.

Here $\h=\mathfrak{sl}_3$, $\h_0=\mathfrak{gl}_2$, and $\gN(\h_1)$ consists of four $H_0$-orbits of
dimension $3,2,2,0$. Hence  $\codim_{\gN(\h_1)} \gN(\h_1)_{sg} = 1$,
$\codim_{\pi^{-1}(\bar\xi)}\pi^{-1}(\bar\xi)_{sg}= 1$, and $\mathfrak Z^{(j)}$ is
%this gives rise to the subvariety $\tilde Z:=\bigsqcup_{\bar\xi\in\bar\Omega_1^{(j)}}(\g_0^*\times \pi^{-1}(\bar\xi)_{sg})$ 
of codimension~$2$ in $\ka^*$, which resembles the bad case of part~{\sf (I)}. But, unlike that 
situation, here $\mathfrak Z^{(j)}$ does contain $K$-regular elements. Recall that 
$\xi\in\pi^{-1}(\bar\xi)\cap\ce$, $\h=[\g_\xi,\g_\xi]$, and $\z$ is the centre of $\g_\xi$. 
Since $\dim \gN(\h_1)=3$, the orbit $G_0{\cdot}\xi$ is of codimension~$3$ in $\pi^{-1}(\bar\xi)$.
Therefore, $\g_\xi=\mathfrak{sl}_3\dotplus\te_{2n-2}$ and 
$\g_{0,\xi}=\mathfrak{gl}_2\dotplus\te_{n-1}$, i.e., $\z=\te_{2n-2}$ and $\z_0=\te_{n-1}$.
Let $\nu\in \gN(\h_1)$ belong to a two-dimensional $H_0$-orbit. Then 
$G_0{\cdot}(\xi+\nu)$ is of codimension~$1$ in $\pi^{-1}(\bar\xi)$, i.e., 
$\codim_{\g_1^*}G_0{\cdot}(\xi+\nu)=n+1$.
Here $\h_{0,\nu}$ is the $2$-dimensional non-abelian Lie algebra, hence $\ind \h_{0,\nu}=0$.
Since 
$\g_{0,\xi+\nu}=\z_0\dotplus \h_{0,\nu}=\te_{n-1}\dotplus \h_{0,\nu}$, we have 
$\ind \g_{0,\xi+\nu}=(n-1)+\ind \h_{0,\nu}=n-1$. Using Eq.~\eqref{eq:min-codim}, we conclude that
$\psi^{-1}(G_0{\cdot}(\xi+\nu))$ contains $K$-orbits of codimension $(n+1)+(n-1)=2n=
\rk\g=\ind\ka$.

Finally, we notice that if $y$ belongs to the $3$-dimensional $H_0$-orbit in $\gN(\h_1)$, then
$\h_{0,y}$ is $1$-dimensional and abelian. Therefore,
$\h_{0,y}$ has the {\sl codim}--$2$ property, which guarantees the ``good behaviour'' of $\mathfrak Y^{(j)}$.

Thus, the strata occurring in part~{\sf (III)} provides no obstacles for the {\sl codim}--$3$ property.

{\color{MIXT}$(\Xi_2)$-case.} \ 
%For $\bar\xi\in\bar\Omega_2$ and  $\xi\in\Omega_2\cap\ce$, the corresponding reduced sub-symmetric pair is of rank~$2$.
Here we only have to prove that $(\Xi_2)_{sg}$ has smaller dimension than $\Xi_2$.
Since
\[
   \Xi_2=\g_0^*\times\Omega_2=\bigsqcup_{\bar\xi\in\bar\Omega_2}(\g_0^*\times \pi^{-1}(\bar\xi)) 
\]
%Set $\pi^{-1}(\bar\xi)_{sg}=\pi^{-1}(\bar\xi)\setminus \pi^{-1}(\bar\xi)_{reg}$. 
and each irreducible component of $\pi^{-1}(\bar\xi)$ contains $G_0$-regular elements, we 
conclude using Eq.~\eqref{eq:fibre-soderzhit} that the set of $K$-regular elements is dense in 
$\Xi_2$. Thus, the codimension~$2$ strata  cause no harm with respect to the 
{\sl codim}--$3$ property.

Thus, the {\sl codim}--$3$ property for $\ka$ fails if and only if $\sat$ has a 
trivial node.
\end{proof}

\begin{rmk}           \label{rem:curious-cod3}
The  symmetric pair $(\g,\g_0)=(\mathfrak{sl}_{2n+1}, \sln\dotplus\slno\dotplus\te_1)$ 
provides a curious and unique example such that the complement of the set of $G_0$-regular
points in $\g_1^*$ contains a component of codimension two in $\g_1^*$, but nevertheless 
$\ka$ possesses the {\sl codim}--$3$ property. 
\end{rmk}

%%%%%%%%%%%%%%%%%%%%%%%%%%%
\section{Concluding remarks and open problems}
\label{sect:finish}

\noindent
We have given a description of maximal commutative subalgebras of the Poisson algebra
$\cs(\ka)=\bbk[\ka^*]$ in the following two cases:

1)  \ If $(\g,\g_0)$ is $\gN$-regular, i.e., $\sat$ contains no black nodes, then $\cs(\ka)^{\g_1}=\bbk[\ka^*]^{\g_1}$
is a maximal commutative subalgebra (Theorem~\ref{thm:N-reg}).

2) \ If  $\sat$ has no trivial nodes, then the argument shift method provides maximal commutative
subalgebras $\cf_\xi(\zk)$ for $\xi\in\ka^*_{reg}$ (Theorems~\ref{thm:cod3}(ii) and \ref{thm:main}). 

The list of remaining symmetric pairs with $\g$ simple consists of the following items:
\begin{enumerate}
\item \ $(\mathfrak{so}_{4n}, \mathfrak{gl}_{2n})$, $n\ge 2$;
\item \ $(\mathfrak{so}_{n+m}, \mathfrak{so}_{m}\dotplus \son)$, $n\ge m>1,\ n-m\ge 3$;
\item \ $(\eus E_7, \eus E_6\dotplus\te_1)$;
\item \ $(\eus E_7, \mathfrak{so}_{12}\dotplus\tri)$;
\item \ $(\eus E_8, \eus E_7\dotplus\tri)$.
\end{enumerate}
For these symmetric pairs, no maximal commutative subalgebras of $\cs(\ka)$ is
known. Of course, $\cf_\xi(\zk)$ is always  of maximal dimension,
since $\ka$ has the {\sl codim}--$2$ property. But the maximality can fail; it does fail %at least 
for the involutions of maximal rank and
$\xi\in \g_1^*\cap\ka^*_{reg}$, see \cite[Example\,4.1]{do}.

On the other hand, there are symmetric pairs, where both above constructions apply. For 
$(\mathfrak{sl}_{2n+1}, \sln\dotplus\slno\dotplus\te_1)$ and $(\h\dotplus\h,\Delta_\h)$, the Satake 
diagram contains neither black nor trivial nodes. Here the maximal commutative
subalgebras $\cf_\xi(\zk)$ and $\cs(\ka)^{\g_1}$ are quite different. Indeed, both algebras are graded polynomial, but the degrees of free homogeneous generators differ considerably. 

For any $\xi\in\g^*_{reg}$, the maximal commutative subalgebra $\cf_\xi(\zg)$ can be lifted to 
$\eus U(\g)$  \cite[Theorem\,3.14]{FFT10} (for the regular {\sl semisimple\/} elements 
$\xi$, the result was earlier obtained in \cite{ryb06}). 

\begin{qtn} %\leavevmode\par
(1) \ Is it possible to quantise (=\,lift to $\eus U(\ka)$) commutative subalgebras of the form  $\cf_\xi(\zk)$, 
$\xi\in\ka^*_{reg}$, \ for\/ $\ka=\g_0\ltimes\g_1$? 
\\  \indent
(2) \ Suppose that\/ $\ka$ has the {\sl codim}--$3$ property. Is it true that any maximal commutative subalgebra of maximal dimension in $\eus S(\ka)$ is necessarily polynomial?
\end{qtn}
It might also be interesting to have some structure results on maximal commutative 
subalgebra of maximal dimension for more general Lie algebras.

\vskip.7ex
{\small {\bf Acknowledgements.}
Both authors were partially supported by the DFG priority programme SPP 1388 "Darstellungstheorie".}


\begin{thebibliography}{33}

\bibitem{leva} {\rusc L.V.~Antonyan}. {\rus O klassifikacii odnorodnykh {e1}lementov 
${\BZ}_2$-graduirovannykh poluprostykh algebr Li},
{\rusi  Vestnik Mosk. Un-ta, Ser. Matem. Meh.}, {\rus N0}\,2 (1982), 29--34 (Russian). 
English translation: {\sc L.V.~Antonyan}. On classification of
homogeneous elements of ${\BZ}_2$-graded semisimple Lie algebras,
{\it  Moscow Univ. Math. Bulletin}, {\bf 37}, {\rus N0}\,2 (1982), 36--43.

\bibitem{bol2}
{\sc A.V.~Bolsinov}. Commutative families of functions related to consistent Poisson brackets, 
{\it Acta Appl. Math.}, {\bf 24}, no.\,3 (1991), 253--274.

\bibitem{charb-mor}
{\sc J.-Y.~Charbonnel} and {\sc A.~Moreau}.
The index of centralizers of elements of reductive Lie algebras, 
{\it Doc. Math.},  {\bf 15}\,(2010), 387--421. 

\bibitem{FFT10} 
{\sc B.~Feigin, E.~Frenkel} and {\sc V.~Toledano Laredo}. 
Gaudin models with irregular singularities, {\it Adv. Math.} {\bf 223}, no.\,3 (2010),  873--948.

\bibitem{jf-TG}
{\sc A.~Joseph} and {\sc D.~Shafrir}.
Polynomiality of invariants, unimodularity and adapted pairs, 
{\it Transform. Groups}, {\bf 15}, no.\,4 (2010), 851--882.

\bibitem{ko63} {\sc B.~Kostant}. Lie group representations on
polynomial rings, {\it  Amer. J. Math.}, {\bf 85}\,(1963), 327--404.

\bibitem{kr71}
{\sc B.\,Kostant} and {\sc S.\,Rallis}. Orbits and representations associated with 
symmetric spaces, {\it Amer. J. Math.}, {\bf 93}\,(1971), 753--809.

\bibitem{mem33} {\sc D.~Luna}. Slices \'etales, {\it Bull. Soc. Math. France},
Memoire~{\bf 33}\,(1973), 81--105.

\bibitem{mf3}  {\rusc A.S.~Miwenko, A.T.~Fomenko}. {\rus Uravneniya {E1}{\u\i}lera
na koneqnomernykh gruppakh Li},  {\rusi Izv. AN SSSR. Ser. Matem.} {\bf 42}, 
{\rus N0}\,2 (1978), 396--415 (Russian).
English translation:
{\sc A.S.\,Mishchenko} and {\sc A.T.\,Fomenko}. Euler equation on finite-dimensional Lie 
groups, {\it Math. USSR--Izv.} {\bf 12}\,(1978), 371--389. 

\bibitem{fonya00}
{\sc A.~Ooms}. On certain maximal subfields in the quotient division ring of an enveloping algebra, 
{\it J.~Algebra}, {\bf 230}, no.\,2 (2000),  694--712.

\bibitem{fonya09}
{\sc A.~Ooms}. Computing invariants and semi-invariants by means of Frobenius Lie algebras,
{\it J. Algebra}, {\bf 321}\,(2009), 1293--1312. 

\bibitem{fonya-et}
{\sc A.~Ooms} and {\sc M.~Van~den~Bergh}. 
A degree inequality for Lie algebras with a regular Poisson semi-center, 
{\it J. Algebra}, {\bf 323}, no.\,2 (2010), 305--322.

\bibitem{rims07} {\sc D.~Panyushev}.
Semi-direct products of Lie algebras and  their invariants,
{\it Publ. RIMS}, {\bf 43}, no.\,4 (2007), 1199--1257.

\bibitem{coadj} {\sc D.~Panyushev}.
On the coadjoint representation of $\mathbb Z_2$-contractions of reductive Lie algebras,
{\it Adv. Math.}, {\bf 213}\,(2007), 380--404.

\bibitem{ppy} {\sc D.~Panyushev}, {\sc A.~Premet} and {\sc O.~Yakimova}. 
On symmetric invariants of centralisers in reductive  Lie algebras,
{\it J. Algebra}, {\bf 313}\,(2007), 343--391.

\bibitem{gnib} {\sc D.~Panyushev} and {\sc O.~Yakimova}. 
The index of representations associated with stabilisers, 
{\it J. Algebra}, {\bf 302}\,(2006), 280--304.

\bibitem{comm-var-sym} {\sc D.~Panyushev} and {\sc O.~Yakimova}. 
Symmetric pairs and associated commuting varieties, 
{\it Math. Proc. Camb. Phil. Soc.}, {\bf 143}, Part\,2  (2007), 307--321. 

\bibitem{do} {\sc D.~Panyushev} and {\sc O.~Yakimova}. 
The argument shift method and maximal commutative subalgebras of Poisson algebras,
{\it Math. Res. Letters}, {\bf 15}, no.\,2 (2008), 239--249.

\bibitem{rais} {\sc M.~Ra\"\i s}. L'indice des produits semi-directs
$E\times_{\rho}\g$, {\it  C.R. Acad. Sc. Paris, Ser. A} \ t.{\bf 287}\,(1978),
195--197.

\bibitem{ryb06} {\rusc L.G.~Rybnikov}.
{\rus Metod sdviga invariantov i model{\cprime} Godena}, {\rusi Funkc. analiz i prilozh.},
{\rus t.}\,{\bf 40}, {\rus vyp}.\,3 (2006), 30--43 (Russian).
English translation: {\sc L.~Rybnikov}. The shift of invariants method and the Gaudin model,
{\it Funct. Anal. Appl.}, {\bf 40}, no.\,3 (2006), 188--199.

\bibitem{sadetov} {\rusc S.T.~Sad{e1}tov}. {\rus Dokazatel{\cprime}stvo gipotezy Miwenko-Fomenko 
(1981)}, {\rusi DAN}, {\bf 397}, {\rus N0}\,6 (2004), 751--754 (Russian).
English translation: {\sc S.~Sadetov}. A proof of the Mishchenko-Fomenko conjecture,
{\it Doklady Mathematics}, {\bf 70}, no.\,1 (2004), 635--638.

\bibitem{slo}  {\sc P.~Slodowy}. Der Scheibensatz f\"ur algebraische 
Transformationsgruppen,  in: {\it ``Algebraische Transformationsgruppen und Invariantentheorie''},  89--113, DMV Sem., Bd.\,{\bf 13}, Basel: Birkh\"auser,  1989.

\bibitem{tar1} {\rusc A.~Tarasov}.
{\rus Maksimal{\cprime}nost{\cprime}\ nekotoryh kommutativnyh podalgebr v algebrah Puassona
polu\-pros\-tyh algebr Li}, {\rusi UMN}, {\bf 57}, {\rus N0}\,5 (2002), 165--166 (Russian). 
English translation: {\sc A.~Tarasov}. The maximality of certain commutative subalgebras in Poisson algebras of a semisimple Lie algebra,
{\it Russian Math. Surveys}, {\bf 57}, no.\,5, (2002), 1013--1014.

\bibitem{t41} {\rusc {E1}.B.~Vinberg, V.V.~Gorbatseviq, A.L.~Oniwik}.
``{\rusi Gruppy i algebry Li - 3}'', {\rus  Sovrem. probl. matematiki. Fundam. napravl., t.\,41.
Moskva: VINITI} 1990 (Russian).
English translation: {\sc V.V.\,Gorbatsevich, A.L.\,Onishchik} and {\sc E.B.\,Vinberg}.
``Lie Groups and Lie Algebras III'' (Encyclopaedia Math. Sci., vol.~41) Berlin: Springer 1994.

%\bibitem{vo}  {\rusc {E1}.B.~Vinberg, A.L.~Oniwik}. {\rusi
%Seminar po gruppam Li i algeb\-rai\-qes\-kim gruppam}. 
%{\rus Moskva: ``Nauka"} 1988 (Russian). English translation: {\sc A.L.~Onishchik} and 
%{\sc E.B.~Vinberg}. ``Lie groups and algebraic groups'', Berlin: Springer, 1990.

\bibitem{kys-1param}  {\sc O.~Yakimova}.  
One-parameter contractions of Lie-Poisson brackets, 
{\it J. Eur. Math. Soc.}, (2014), to appear (=\textsf{arXiv: 1202.3009}). 

\bibitem{kys-4cases}  {\sc O.~Yakimova}. On symmetric invariants of $\BZ_2$-contractions and 
other semi-direct products,
{\it In preparation}.  
\end{thebibliography}
\end{document}